\documentclass[12pt,reqno]{amsart}

\usepackage{times}
\usepackage[T1]{fontenc}
\usepackage{mathrsfs}
\usepackage{latexsym}
\usepackage[dvips]{graphics}
\usepackage[titletoc,title]{appendix}
\usepackage{epsfig}

\usepackage{amsmath,amsfonts,amsthm,amssymb,amscd}
\input amssym.def
\input amssym.tex
\usepackage{color}
\usepackage{hyperref}

\usepackage{color}
\usepackage{breakurl}
\usepackage{comment}
\newcommand{\bburl}[1]{\textcolor{blue}{\url{#1}}}

\newcommand{\be}{\begin{equation}}
\newcommand{\ee}{\end{equation}}
\newcommand{\bea}{\begin{eqnarray}}
\newcommand{\eea}{\end{eqnarray}}

\newtheorem{thm}{Theorem}[section]

\newtheorem{cor}[thm]{Corollary}
\newtheorem{lem}[thm]{Lemma}

\newtheorem{defi}[thm]{Definition}

\newtheorem{rek}[thm]{Remark}

\numberwithin{equation}{section}

\newcommand{\Mod}[1]{\ \mathrm{mod}\ #1}

\begin{document}

\title{A New Transcendental Number from $N^N$}

\author{H\`ung Vi\d{\^e}t Chu}
\email{\textcolor{blue}{\href{mailto:chuh19@mail.wlu.edu}{chuh19@mail.wlu.edu}}}
\address{Department of Mathematics, University of Illinois at Urbana-Champaign, Urbana, IL 61820, USA}

\begin{abstract}
We first give a summary of the history of transcendental numbers then use a nice technique by G. Dresden to prove a new transcendental number. In particular, while previous work looked at the last non-zero digit of $n^n$, we consider the digit right before its last non-zero digit and show that the infinite decimal built from these digits is transcendental. 
\end{abstract}

\maketitle
\section{Introduction and main result}
We begin by giving the definition of a transcendental number and its brief history. 
\begin{defi}
An \textbf{algebraic} number is a complex number, which is a root of a non-zero polynomial equation with integer coefficients. A number is said to be \textbf{transcendental} if it is not algebraic. 
\end{defi}
The name \textit{transcendental} comes from the Latin \textit{transcend\u{e}re} that means to transgress or to go over. It was generally believed that Euler was the first to define transcendental numbers in the modern sense. 
Though much less is known about transcendental numbers than algebraic numbers, transcendental numbers are not rare. In fact, Cantor \cite{Can} proved that most numbers are transcendental because while the set of complex numbers is uncountable, the set of algebraic numbers is only countable. In 1844, Liouville constructed and proved the transcendence of
$$\sum_{k=1}^{\infty}10^{-k!}\ =\ 0.11000100000000000000000100\ldots ,$$
which is now known as the Liouville constant, the first transcendental number to be proved. 

Later, Hermite and Lindemann showed the transcendence of $e$ and $\pi$, respectively \cite{Lin}. It is worth noting: in 1837, Wantzel \cite{Wa} proved that lengths constructible with ruler and compass must be algebraic. This result is a breakthrough towards showing the impossibility of squaring the circle, a famous problem dated back to the ancient Greeks; therefore, the later Lindemann's proof of the transcendence of $\pi$ stopped fruitless attempts of circle squarers. 

In 1934, the Gelfond-Schneider theorem made another breakthrough on transcendental numbers. It states that if $a$ and $b$ are algebraic numbers with $a\neq  0, 1,$ and $b$ irrational, then any value of $a^b$ is a transcendental number. Nowadays, much is still unknown about transcendental numbers. For example, the following are possible transcendental numbers not yet (dis)proved: $\pi\pm e$, $\pi\cdot e$, $\pi/e$, $\pi^e$, the Riemann zeta function at odd integers, the Euler-Mascheroni constant $\gamma$, and this list goes on.

We are now ready to look at our numbers of concern. First, consider the last digit of each term in sequence $\{(n^n)\}_{n=1}^\infty$. We construct the number $N = 0.d_1d_2d_3\ldots$, where $d_n$ is the last digit of $n^n$. For example, since $1^1 = 1$, $2^2 = 4$, $3^3 = 27$, $4^4 = 256$, and so on, we have $N = 0.1476563690\ldots$. In 1998, Euler and Sadek \cite{ES} proved that $N$ is rational with a period of $20$; that is,
\begin{align*}N &\ =\ 0.\overline{14765636901636567490}\\ &\ =\ (1476563690\mbox{ }1636567490)/(9999999999\mbox{ }9999999999).\end{align*}
Though surprising, the result is not hard to prove by noticing that $$n^n \ \equiv\ (n+20k)^{n+20k}\Mod 10.$$
Later, Dresden \cite{G1} extended this nice result by looking at the numbers formed by the last \textit{non-zero} digits of $n^n$ and $n!$ and proved that they are both irrational. This is an interesting and somewhat counter-intuitive result because ignoring $0$ as the last digit changes the number radically.\footnote{Dresden looked at the last non-zero digit of $n!$ because after $n=4$, the last digit of $n!$ is always zero. In a letter to Mathematics Magazine (February 2002), Stan Wagon pointed out that the question of the periodicity of the last non-zero digit of $n!$ 
appeared several times in Crux Mathematicorum in the 1990s.} Dresden also raised the question of whether they were transcendental and answered this positively in a later paper \cite{G2}. The proof used an ingenious technique involving the below theorem that characterizes transcendental numbers. 
\begin{thm}[Thue, Siegel, Roth]\footnote{This theorem is built from the work of three mathematicians and is optimal in the sense that if we drop the $\varepsilon$, we theorem no longer holds.}\label{TSR}
For $\alpha$ algebraic and $\varepsilon > 0$, there exist only finitely many rational numbers $p/q$ such that $$\bigg|\alpha-\frac{p}{q}\bigg| \ < \ \frac{1}{q^{2+\varepsilon}}.$$
\end{thm}
Thus, if we can show that for some fixed $\varepsilon > 0$, there are infinitely many such numbers $p/q$ that satisfy the above inequality, then $\alpha$ is transcendental. Since then, research on numbers formed by the last non-zero digits has made considerable progress. In particular, Grau and Oller-Marcen \cite{Gr} studied the sequences defined by the last and the last non-zero digits of $n^n$ in base $b$. Ikeda and Matsuoka \cite{IM} generalized Dresden's technique and found new examples of transcendental numbers formed by $n^n, n^{n^n}, n^{n^{n^n}}$, etc. and the number of trailing zeros of $n^j$ for any fixed $j\in \mathbb{N}$.\footnote{Trailing zeros are zeros at the end of a number (if any).} Last, but not least, Deshouillers and Ruzsa \cite{DR} employed Dresden's technique to prove an asymptotic density of a certain sequence. However, none has looked at the digit right before the last non-zero digit. We call this digit \textit{rbln} (right before the last non-zero). 

Our main result concerns this natural extension; that is, we work with the number $P = 0.d_1d_2d_3\ldots$, where $d_n$ is the rbln digit of $n^n$. In order that $P$ is well-defined, if $n^n$ has exactly $1$ non-zero digit, we will take $0$ as the rbln digit. Formally, let $n\in\mathbb{N}$, then the function $\mbox{rbln}(n)$ 
\begin{enumerate}
    \item gives $0$ if $n$ has exactly one non-zero digit,
    \item gives the rbln digit of $n$ if $n$ has at least two non-zero digits. 
\end{enumerate} 
For example, $\mbox{rbln}(1000) = 0$ and $\mbox{rbln}(10504200) = 4$. Our definition of the function $\mbox{rbln}$ is natural since $1000 = 01000$, and so, $\mbox{rbln}(1000) = \mbox{rbln}(01000) = 0$. On the other hand, the function $\mbox{lnzd}(n)$ is defined as in \cite{G1}; that is, $\mbox{lnzd}(n)$ gives the last non-zero digit of $n$. 
\begin{thm}\label{pen}
Let $P = 0.d_1d_2d_3\ldots$, where $d_n = \mbox{rbln}(n^n)$. Then $P$ is transcendental. 
\end{thm}
To prove Theorem \ref{pen}, we use Dresden's technique \cite{G2}. However, because we consider the rbln digits, our proof requires an argument about the invariance of $\mbox{rbln}(n^n)$ across certain values of $n$ (see Lemma \ref{inva} and Lemma \ref{2inva} below). 

\section{Proof of Theorem \ref{pen}}
\subsection{Preliminaries}
The following lemma is very useful in proving the transcendence of $P$. It shows the relationship between $\mbox{lnzd}(n^n)$ and $\mbox{rbln}(n^n)$.
\begin{lem}\label{inva}
Let $n\in \mathbb{N}$ and $100\mid n$. Then 
\begin{enumerate}
\item $\mbox{lnzd}(n^n) =  6$ implies that $\mbox{rbln}(n^n) = 7$,
\item $\mbox{lnzd}(n^n) = 5$ implies that $\mbox{rbln}(n^n) = 2$,
\item $\mbox{lnzd}(n^n) = 1$ implies that $\mbox{rbln}(n^n) = 0$.
\end{enumerate}
\end{lem}
\begin{rek}\label{skey}\normalfont
Since $76^2$ and $25^2$ end with $76$ and $25$, respectively, if we let $N$ be a number with last two digits $76$ (25, resp), then $N^j$ also end with $76$ (25, resp) for all $j\in\mathbb{N}$.
\end{rek}
\begin{proof}
First, we prove item (1). In order that $\mbox{lnzd}(n^n) = 6$, $\mbox{lnzd}(n) \in \{2,4,6,8\}$. We consider the case when $\mbox{lnzd}(n) = 2$. The other cases are similar. Let $n'$ be the number $n$ without its trailing zeros and let $i\ge 2$ be the largest integer such that $10^i\mid n$. We have $$\mbox{rbln}(n^n) \ =\ \mbox{rbln}((n'\cdot 10^i)^{n'\cdot 10^i}) \ =\ \mbox{rbln}((n')^{n'\cdot 10^i}) \ =\ \mbox{rbln}((n'^{100})^{n'\cdot 10^{i-2}}).$$ If $\mbox{rbln}(n') = 0$, then since $2^{100}$ ends in $76$, $n'^{100}$ ends in $76$. By Remark \ref{skey}, $(n'^{100})^{n'\cdot 10^{i-2}}$ also ends in $76$, which implies that $\mbox{rbln}(n^n) = 7$. If $\mbox{rbln}(n')\neq 0$, note that $$12^{100}, 22^{100}, \ldots, 92^{100}$$ all end in $76$. Again, we have $\mbox{rbln}(n^n) = 7$.

The proof of item (2) is similar to the proof of item (1), so we omit it. 

We now prove item (3). Because $\mbox{lnzd}(n^n)=1$, $d = \mbox{lnzd}(n)=\mbox{lnzd}(n')\in \{1,3,7,9\}$. So, $\mbox{lnzd}(n^{100})=1$. Let $k=\mbox{rbln}(n^{100})$ and $\ell=\mbox{rbln}(n')$. Then $k$ is determined solely by the integer $\ell d$. Since $d\in \{1,3,7,9\}$, $\ell d$ is divisible by neither $2$ nor $5$. A quick check for all possible numbers $\ell d$ in $\{1,2,\ldots,99\}$ by computer shows that $(\ell d)^{100}$ ends in $01$. Hence, $k = 0$. Since $100\mid n$, we know that $\mbox{rbln}(n^n)=0$. 
\end{proof}
\begin{cor}\label{cG1} Let $n\in \mathbb{N}$ and $100\mid n$. The following claims are true. 
\begin{enumerate}
\item If $\mbox{lnzd}(n)$ is in $\{2,4,6,8\}$, then $\mbox{rbln}(n^n) = 7$.
\item If $\mbox{lnzd}(n) = 5$, then $\mbox{rbln}(n^n) = 2$.
\item If $\mbox{lnzd}(n)$ is in $\{1,3,7,9\}$, then $\mbox{rbln}(n^n) = 0$.
\end{enumerate}
\end{cor}
\begin{proof}
The proof follows directly from the proof of Lemma \ref{inva}.
\end{proof}
We now prove a lemma parallel to \cite[Lemma 2]{G2}.
\begin{lem}\label{2inva}
Let $n\in\mathbb{N}$ and $100\nmid n$, then 
$$\mbox{rbln}(n^n) \ =\ \mbox{rbln}((n+100)^{n+100}).$$
\end{lem}
\begin{proof}
Because $n$ is not divisible by $100$, either the last digit of $n$ is non-zero or the next-to-last digit of $n$ is non-zero. We consider the two corresponding cases. 

Case 1: the last digit of $n$ is non-zero. Call this digit $d$. Then the last digits of $n^n$ and $(n+100)^{n+100}$ are also non-zero. Hence, it suffices to prove that 
$$n^n \ \equiv\ (n+100)^{n+100}\Mod 100.$$
Equivalently, 
$$n^n \ \equiv\ n^n\cdot n^{100}\Mod 100\mbox{ } \mbox{, so  }\mbox{ } n^n(n^{100}-1)\equiv 0\Mod 100.$$

If $d$ is in  $\{2,4,6,8\}$, then $4\mid n^n$ and $n^{100}$ ends in $76$ as in the proof of Lemma \ref{inva} item (1). So, $75\mid(n^{100}-1)$ and so, $100\mid n^n(n^{100}-1)$, as desired. 

If $d$ is in $\{1,3,7,9\}$, then $n^{100}$ ends in $01$ as in the proof of Lemma \ref{inva} item (3). So, $100\mid(n^{100}-1)$ and so, $100\mid n^n(n^{100}-1)$, as desired.

If $d = 5$, then $n^{100}$ ends in $25$ as implied by Lemma \ref{inva} item (2). So, $n^{100}-1$ ends in $24$ and thus, $4\mid(n^{100}-1)$. On the other hand, since $5\mid n$, $25\mid n^n$. Hence, $100\mid n^n(n^{100}-1)$, as desired.

Case 2: the last digit of $n$ is zero and the next-to-last digit, called $d$, is non-zero. Let $n' = n/10$. Note that $\mbox{lnzd}(n') = d$. It suffices to prove that 
$$n'^n \ \equiv \ (n'+10)^{n+100} \Mod 100.$$
Equivalently, $$n'^n \ \equiv\ n'^nn'^{100} \Mod 100\mbox{ }\mbox{, so }\mbox{ }n'^n(n'^{100}-1)\equiv 0\Mod 100.$$
The same argument as in Case 1 would show that the last congruence holds. 

We have completed the proof. 
\end{proof}

\subsection{Proof of Theorem \ref{pen}}
Recall that we form our number $P = 0.d_1d_2d_3\ldots$ with $d_n = \mbox{rbln}(n^n)$. In particular, 
\begin{align*}P \ =\ 0.&\mbox{ }0025254180\mbox{ }1551717277\mbox{ }2867270364\\
&\mbox{ }3713731057\mbox{ }4609296542\mbox{ }5975755837\ldots.\end{align*}
At this point, we hardly see any pattern. To prove transcendence, we will show that there are infinitely many rational numbers $p_n/q_n$ such that 
$$\bigg|P - \frac{p_n}{q_n}\bigg| \ < \ \frac{1}{q_n^{2.4}}.$$
Then, by Theorem \ref{TSR}, $P$ is transcendental. To discover a pattern, let's create a sequence of numbers $P_0 = P, P_1, P_2, \ldots$ such that each $P_n$ is formed by replacing every digit in $P_{n-1}$ with zeros except for every tenth  digit (which will be left alone). The following picture illustrates the process
\begin{align*}
    &P_0 \ =\ 0.\mbox{ }0025254180\mbox{ }1551717277\mbox{ }2867270364\mbox{ }3713731057\mbox{ }4609296542\mbox{ }\ldots\\
    &P_1 \ =\ 0.\mbox{ }0000000000\mbox{ }0000000007\mbox{ }0000000004\mbox{ }0000000007\mbox{ }0000000002\ldots\\
    &P_2 \ =\ 0.\mbox{ }0000000000\mbox{ }0000000000\mbox{ }0000000000\mbox{ }0000000000\mbox{ }0000000000\ldots
\end{align*}
The number $P_1$ keeps the digits of $P_0$ at the $10^{th}, 20^{th}, 30^{th}$, etc. decimal places. The number $P_2$ keeps the digits of $P_0$ at the $100^{th}, 200^{th}, 300^{th}$, etc. places. In general, the number $P_n$ keeps the digits of $P_0$ at the $10^n, 2\cdot 10^n, 3\cdot 10^n$, and so on. Let's look at the digits we keep by condensing all the replacements of $0$ in the above construction, but distances between digits still corresponds to the number of zeros between them. We observe a beautiful pattern after $P_2$
\begin{align*}
    &P_0 \ =\ 0.0025254180\mbox{ }\ldots\\
    &P_1 \ =\ 0.\mbox{ }0\mbox{ }7\mbox{ }4\mbox{ }7\mbox{ }2\mbox{ }7\mbox{ }4\mbox{ }7\mbox{ }0\mbox{ }0\mbox{ }\ldots\\
    &P_2 \ =\ 0.\mbox{ }\mbox{ }0\mbox{ }\mbox{ }7\mbox{ }\mbox{ }0\mbox{ }\mbox{ }7\mbox{ }\mbox{ }2\mbox{ }\mbox{ }7\mbox{ }\mbox{ }0\mbox{ }\mbox{ }7\mbox{ }\mbox{ }0\mbox{ }\mbox{ }0\mbox{ }\mbox{ }\mbox{ }\ldots\\
    &P_3 \ =\ 0.\mbox{ }\mbox{ }\mbox{ }0\mbox{ }\mbox{ }\mbox{ }7\mbox{ }\mbox{ }\mbox{ }0\mbox{ }\mbox{ }\mbox{ }7\mbox{ }\mbox{ }\mbox{ }2\mbox{ }\mbox{ }\mbox{ }7\mbox{ }\mbox{ }\mbox{ }0\mbox{ }\mbox{ }\mbox{ }7\mbox{ }\mbox{ }\mbox{ }0\mbox{ }\mbox{ }\mbox{ }0\mbox{ }\mbox{ }\mbox{ }\ldots\\
    &P_4 \ =\ 0.\mbox{ }\mbox{ }\mbox{ }\mbox{ }0\mbox{ }\mbox{ }\mbox{ }\mbox{ }7\mbox{ }\mbox{ }\mbox{ }\mbox{ }0\mbox{ }\mbox{ }\mbox{ }\mbox{ }7\mbox{ }\mbox{ }\mbox{ }\mbox{ }2\mbox{ }\mbox{ }\mbox{ }\mbox{ }7\mbox{ }\mbox{ }\mbox{ }\mbox{ }0\mbox{ }\mbox{ }\mbox{ }\mbox{ }7\mbox{ }\mbox{ }\mbox{ }\mbox{ }0\mbox{ }\mbox{ }\mbox{ }\mbox{ }0\mbox{ }\mbox{ }\mbox{ }\mbox{ }\ldots
\end{align*}
By Corollary \ref{cG1}, we know that for all $i>j\ge 2$, the sequence of digits we keep for $P_i$ is the same as the sequence of digits we keep for $P_{j}$. This implies that $R_n = P_n - P_{n+1}$ is rational for $n\ge 2$; the cases $n=0$ and $n=1$ follow immediately from Lemma \ref{2inva}. Observe that for $n\ge 2$, $$P_n \ =\  0.\mbox{ }\mbox{ }\mbox{ }0\mbox{ }\mbox{ }\mbox{ }7\mbox{ }\mbox{ }\mbox{ }0\mbox{ }\mbox{ }\mbox{ }7\mbox{ }\mbox{ }\mbox{ }2\mbox{ }\mbox{ }\mbox{ }7\mbox{ }\mbox{ }\mbox{ }0\mbox{ }\mbox{ }\mbox{ }7\mbox{ }\mbox{ }\mbox{ }0\mbox{ }\mbox{ }\mbox{ }0\mbox{ }\mbox{ }\mbox{ }\ldots$$ is quite close to the rational number
$$\frac{s_n}{t_n} \ =\ 0.\mbox{ }\mbox{ }\mbox{ }0\mbox{ }\mbox{ }\mbox{ }7\mbox{ }\mbox{ }\mbox{ }0\mbox{ }\mbox{ }\mbox{ }7\mbox{ }\mbox{ }\mbox{ }0\mbox{ }\mbox{ }\mbox{ }7\mbox{ }\mbox{ }\mbox{ }0\mbox{ }\mbox{ }\mbox{ }7\mbox{ }\mbox{ }\mbox{ }0\mbox{ }\mbox{ }\mbox{ }7\mbox{ }\mbox{ }\mbox{ }\ldots $$
Actually, these non-zero digits of $s_n/t_n$ are $10^n$ decimal places apart. So, $$\frac{s_n}{t_n} \ =\ \frac{7}{10^{2\cdot 10^n}-1}.$$ Also, because $t_n < 10^{2\cdot 10^n}$, we have
$$\bigg|P_n - \frac{s_n}{t_n}\bigg| \ <\ \frac{3}{10^{5\cdot 10^n}}\ <\ \frac{1}{10^{4.8\cdot 10^n}}\ <\ \frac{1}{t_n^{2.4}}.$$
Let's now relate this back to $P$. Since $R_n = P_n-P_{n+1}$ has period  $10^{n+1}$ for $n>0$, we know that each rational number $R_n$ has denominator $10^{10^{n+1}}-1$ for $n > 0$ and denominator $10^{100}-1$ for $n = 0$. For each $n\in \mathbb{N}$, define $p_n/q_n$ as
$$\frac{p_n}{q_n} \ =\ \bigg(\sum_{i=0}^{n-1}R_i\bigg)+\frac{s_n}{t_n}.$$
Because the denominator of each $R_i$ divides $t_n$, we know that $q_n = t_n = 10^{2\cdot 10^n}-1$. 
Finally, 
$$\bigg|P - \frac{p_n}{q_n}\bigg| \ =\ \bigg|P - \bigg(\sum_{i=0}^{n-1}R_i\bigg) - \frac{s_n}{t_n}\bigg|\ =\ \bigg|P_n - \frac{s_n}{t_n}\bigg| \ <\ \frac{1}{t_n^{2.4}} \ =\ \frac{1}{q_n^{2.4}}.$$
By Theorem \ref{TSR}, we know that $P$ is transcendental. 

\section{Future Work}
We end with several questions for future research. First, is there a simple sufficient condition on which sequences produce transcendental numbers? Note that Ikeda and Matsuoka gave a sufficient condition \cite{IM}; however, the condition requires a check for rationality, which is not always easy. Another question is whether the decimal from the last non-zero digits of the triangular numbers transcendental?\footnote{Dresden \cite{G2}} What about their rbln digits? Last, but not least, is the decimals from the last digits of the prime numbers transcendental? \footnote{Euler and Sadek \cite{ES}}

\end{document}